\documentclass[12 pt]{amsart}

\usepackage{amsthm}
\usepackage{amssymb}
\usepackage{latexsym}
\usepackage{url}
\usepackage{graphicx}

\author{Felix Goldberg}
\address{Department of Mathematics, Technion-IIT, Technion City, Haifa 32000, ISRAEL}
\email{felix.goldberg@gmail.com}

\title{A spectral bound for graph irregularity}
\date{August 12, 2013}

\newtheorem{thm}{Theorem}[section]

\newtheorem{lem}[thm]{Lemma}

\begin{document}

\begin{abstract}
The imbalance of an edge $e=\{u,v\}$ in a graph is defined as $i(e)=|d(u)-d(v)|$, where $d(\cdot)$ is the vertex degree. The irregularity $I(G)$ of $G$ is then defined as the sum of imbalances over all edges of $G$. This concept was introduced by Albertson who proved that $I(G) \leq \frac{n^{3}}{27}$ (where $n=|V(G)|$) and obtained stronger bounds for bipartite and triangle-free graphs. Since then a number of additional bounds were given by various authors. In this paper we prove a new upper bound, which improves a bound found by Zhou and Luo in 2011. Our bound involves the Laplacian spectral radius $\lambda$. 
\end{abstract}

\subjclass{05C35,05C50,05C07}

\keywords{irregularity, Laplacian matrix}

\maketitle

\section{Introduction}
Albertson \cite{Alb97} has defined the \emph{irregularity} of a graph $G$ as:
$$
I(G)=\sum_{(u,v) \in E(G)}{|d(u)-d(v)|},
$$
where $d(u)$ is the degree of vertex $u$. Clearly $I(G)$ is zero if and only if $G$ is regular and for non-regular graphs $I(G)$ is a measure of the defect of regularity. Albertson proved the following upper bound:
\begin{equation}\label{eq:alb}
I(G) \leq \frac{4n^{3}}{27}.
\end{equation}

Abdo, Cohen and Dimitrov \cite{AbdCohDim12} improved Albertson's bound:
\begin{equation}\label{eq:acd}
I(G) \leq \lfloor \frac{n}{3} \rfloor \lceil \frac{2n}{3} \rceil (\lceil \frac{2n}{3} \rceil - 1).
\end{equation}

Additional upper bounds on $I(G)$ have been given by various authors: Hansen and M\'{e}lot \cite{HanMel05}, Henning and Rautenbach \cite{HenRau07}, Zhou and Luo \cite{ZhoLuo08}, and Fath-Tabar \cite{FathTabar11}. These bounds are, strictly speaking, noncomparable but the bound of Zhou and Luo seems to be much much sharper than the others for most graphs. We obtain here a new upper bound which is always less than the Zhou-Luo bound or equal to it.

To state the results, let us define the quantity $Z_{G}=\sum_{u \in V(G)}{d(u)^{2}}$. It is sometimes called \emph{the first Zagreb index} of $G$ (cf. \cite{FathTabar11}).

\begin{thm}[{{\cite[Theorem 1]{ZhoLuo08}}}]\label{thm:zl}
Let $G$ be a graph on $n$ vertices and with $m$ edges. Then:
$$
I(G) \leq \sqrt{m(nZ_{G}-4m^{2})}.
$$
\end{thm}

Let us now recall the definition of the Laplacian matrix $L$ of the graph $G=(V,E)$ whose vertices are labelled $\{1,2,\ldots,n\}$:
$$
L_{ij} = \begin{cases}
-1  & \text{, if } (i,j) \in E \\
0   & \text{, if } (i,j) \notin E \text{ and } i \neq j\\
\sum_{k \neq i}{L_{ik}} & \text{, if } i=j.
\end{cases}
$$

It is obvious from the definition that $L$ is a positive semidefinite matrix. Surveys of its variegated and fascinating properties can be found in \cite{Mer94,Moh97,Zha11}. One simple fact will be germane to us here: the largest eigenvalue $\lambda_{\max}$ of $L$ satisfies $\lambda_{\max} \leq n$. 


We can now state our new result which is clearly an improvement upon the Zhou-Luo bound:
\begin{thm}\label{thm:main}
Let $G$ be a graph on $n$ vertices and with $m$ edges. 
Then:
\begin{equation}\label{eq:new}
I(G) \leq \sqrt{m(nZ_{G}-4m^{2})(\lambda_{\max}/n)}.
\end{equation}
\end{thm}

\section{Proof of the main result}
The quadratic form defined by $L$ has the following useful expression (where we identify the vector $x \in \mathbb{R}^{n}$ with a function $x:V(G) \rightarrow \mathbb{R}$):
\begin{equation}\label{eq:xlx}
x^{T}Lx=\sum_{(u,v) \in E(G)}{(x(u)-x(v))^{2}}.
\end{equation}

We also need Fiedler's \cite{Fie75} well-known characterization of $\lambda_{\max}$:
\begin{lem}\label{lem:fiedler}
$$
\lambda_{\max}=2n \max_{x}  \frac{\sum_{(u,v) \in
        E(G)}{(x(u)-x(v))^{2}}} {\sum_{u \in
        V(G)}\sum_{v \in V(G)}{(x(u)-x(v))^{2}}},
$$
where $x$ is a nonconstant vector.
\end{lem}

\begin{proof}[Proof of Theorem \ref{thm:main}]
The first step is to apply the Cauchy-Schwarz inequality:
\begin{equation}\label{eq:cs}
I(G)=\sum_{(u,v) \in E(G)}{|d(u)-d(v)|} \leq \sqrt{m} \sqrt{\sum_{(u,v) \in E(G)}{(d(u)-d(v))^{2}}}.
\end{equation}

In light of \eqref{eq:xlx} we have:
$$
I(G) \leq \sqrt{m}\sqrt{d^{T}Ld}.
$$
We now turn to estimate $d^{T}Ld$ using Lemma \ref{lem:fiedler} and Lagrange's identity:
$$
d^{T}Ld \leq  \frac{\lambda_{\max}}{2n} \sum_{u \in
        V(G)}\sum_{v \in V(G)}{(d(u)-d(v))^{2}}=
$$
$$
=\frac{\lambda_{\max}}{n}[(\sum_{v \in V(G)}{d(v)})^{2}-n\sum_{v \in V(G)}{d(u)^{2}}].
$$
Clearly, the latter expression is equal to 
$$
\frac{\lambda_{\max}}{n}(4m^{2}-nZ_{G}).
$$
\end{proof}

\section{An example}

Consider the \emph{yoke} graph $G=Y_{n_{1},n_{2}}$ which consists of two cycles of lengths $n_{1}$ and $n_{2}$ ($n_{1}+n_{2}=n$), connected by an edge. It is not hard to see that in this case $I(G)=4$. 

\begin{figure}[here]\label{fig:1}
\includegraphics[height=4cm,width=5cm]{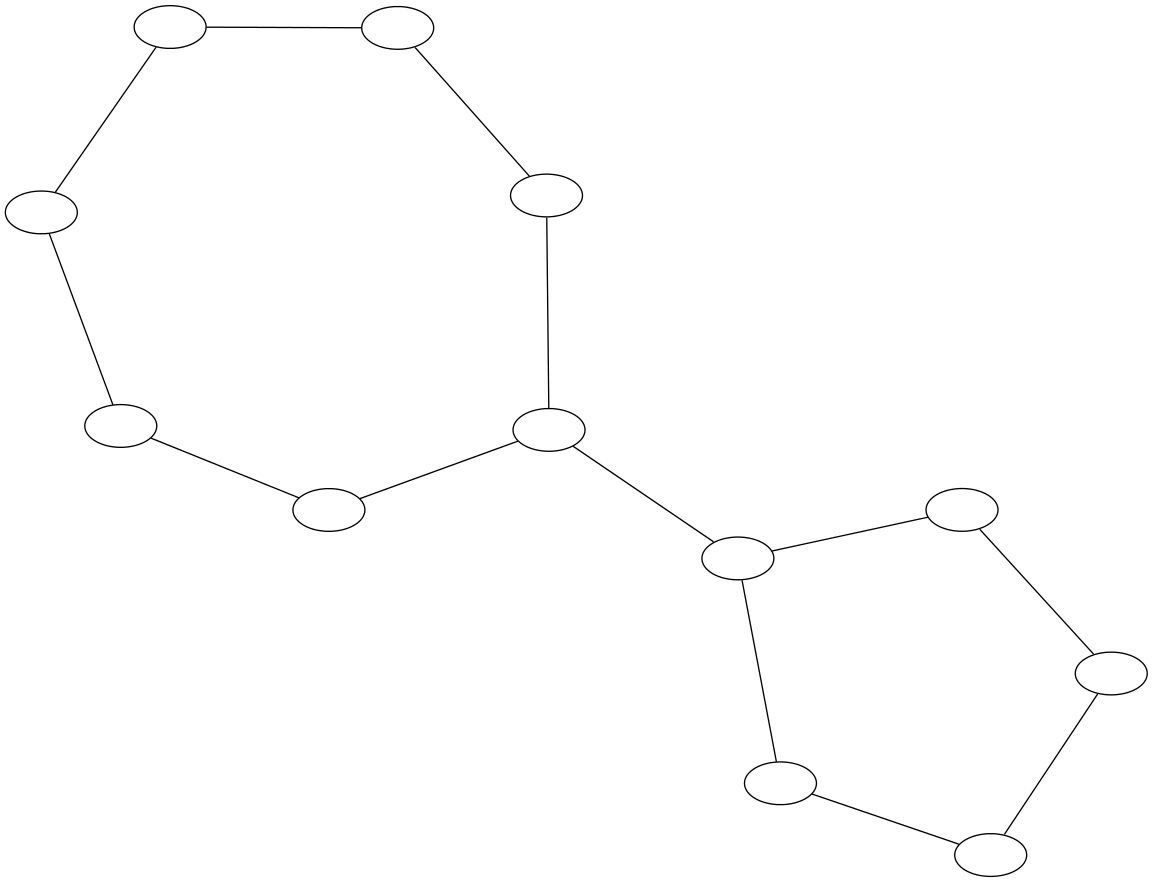} 
\caption{The graph $Y(7,5)$}
\end{figure}

In order to compare the bounds given by Theorems \ref{thm:zl} and \ref{thm:main} we observe that:
$$
m=n+1, Z_{G}=4n+10.
$$

Therefore, the bound given by Theorem \ref{thm:zl} is:
\begin{equation}\label{eq:ex_zl}
I(G) \leq \sqrt{2(n+1)(n-2)}=\Theta(n).
\end{equation}

To estimate $\lambda_{\max}$ we can use a result due to Merris \cite{Merris98b}. To state it, we define $m(v)$ to be the average degree of the neighbours of vertex $v$.
\begin{lem}\cite{Merris98b}\label{lem:merris}
$\lambda_{\max} \leq \max\{d(v)+m(v) \mid v \in V(G)\}$.
\end{lem}
Using this lemma we obtain that $\lambda_{\max} \leq \frac{16}{3}$ for any yoke graph and therefore the bound of Theorem \ref{thm:main} is at least as good as:
\begin{equation}\label{eq:ex_main}
I(G) \leq \sqrt{\frac{32}{3}\frac{(n+1)(n-2)}{n}}=\Theta(\sqrt{n}).
\end{equation}

Obviously, \eqref{eq:ex_main} is much nearer to the true value of $I(G)$ than \eqref{eq:ex_zl}, although it still leaves something to be desired.

\section{Another example - trees}
We wish now to compare Theorem \ref{thm:main} to a specialized result of Zhou and Luo which gives a diffferent and very interesting bound for trees. 
\begin{thm}\cite[Theorem 4]{ZhoLuo08}\label{thm:tree}
Let $T$ be a tree with $p$ pendant vertices. Then $I(G) \leq p(p-1)$.
\end{thm}


Before reporting a comparison between the two bounds, we wish to point out that by an observation of Albertson (\cite[Corollary 5]{Alb97}) $I(G)$ must be an even integer. Therefore, any upper bound on $I(G)$ can be replaced by the largest even integer not exceeding it. 

We have computed the bounds of Theorem \ref{thm:main} (truncated to an even integer, as explained above) and of Theorem \ref{thm:tree} (which always produces even integers) for the $106$ non-isomorphic trees on ten vertices. The summary of the results is:
\begin{itemize}
\item
For $46$ trees Theorem \ref{thm:main} is better than Theorem \ref{thm:tree}.
\item
For $18$ trees both bounds agree.
\item
For $42$ trees Theorem \ref{thm:tree} is better than Theorem \ref{thm:main}.
\end{itemize}

To take some specific examples, consider first the tree in Figure \ref{fig:trees}. It is easy to compute that $I(T_{15})=22$. Since the graph has $7$ pendant vertices, Theorem \ref{thm:tree} yields the estimate $I(T_{15}) \leq 42$. On the other hand, the right-hand side of \eqref{eq:new} is $27.8614$ and so Theorem \ref{thm:main} gives us the estimate $I(G) \leq 26$.
\begin{figure}[here]
\includegraphics[height=4cm,width=5cm]{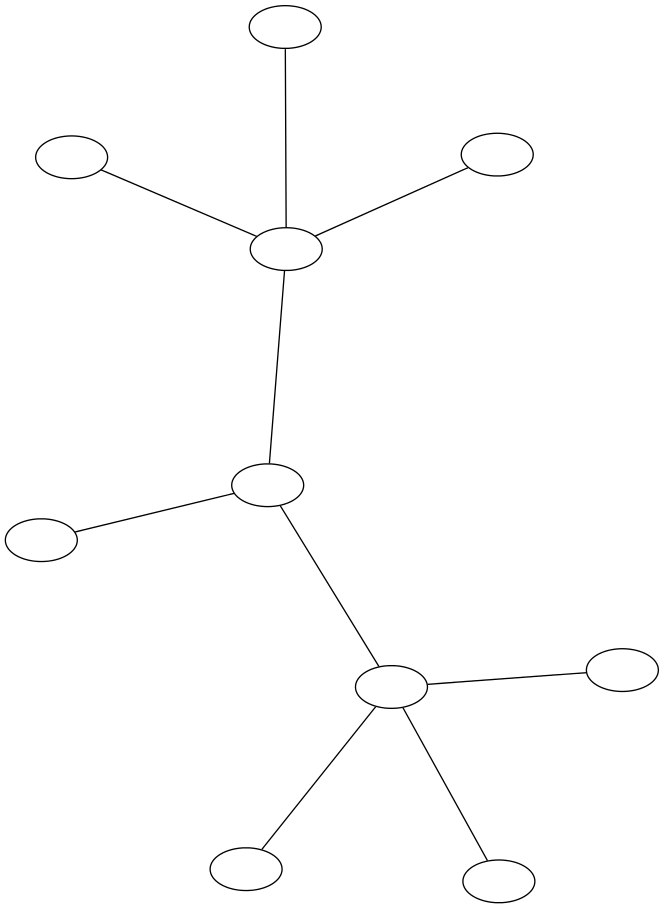} 
\caption{The tree $T_{15}$}\label{fig:trees}
\end{figure}

On the other hand, if we take the path $P_{n}$ on $n$ vertices, then $I(G)=2$ and this value coincides with the bound of Theorem \ref{thm:tree} since $p=2$ in this case. However, Theorem \ref{thm:main} gives a poor estimate in this case - we have $m=n-1,Z_{G}=4n-6$, and $\lambda_{\max}=2(1+\cos{\frac{\pi}{n}})$. Since $\lambda_{\max}$ tends to $4$ as $n$ grows, we can calculate that the right-hand side of \eqref{eq:new} tends to $\sqrt{\frac{8(n-1)(n-2)}{n}}=\Theta(\sqrt{n})$.

\bibliographystyle{abbrv}
\bibliography{nuim}
\end{document}